\documentclass{article}
\usepackage{amsmath,amsthm,amscd,amssymb}
\usepackage[mathscr]{eucal}
\usepackage{amssymb}
\usepackage{latexsym}

\theoremstyle{definition}
\newtheorem{Def}{Definition}[section]
\newtheorem{Thm}[Def]{Theorem}

\newtheorem{Rem}[Def]{Remark}

\newtheorem{Cor}[Def]{Corollary}

\newtheorem{Lem}[Def]{Lemma}

\numberwithin{equation}{section}

\begin{document}
 
\title{Ramanujan type congruences for modular forms of several variables}

\author{Toshiyuki Kikuta and Shoyu Nagaoka}
\maketitle

%\date{}

%\dedicatory{}

%\keywords{}

\maketitle


\begin{abstract}
We give congruences between the Eisenstein series and a cusp form in the cases of Siegel modular forms and Hermitian modular forms. We should emphasize that there is a relation between the existence of a prime dividing the $k-1$-th generalized Bernoulli number and the existence of non-trivial Hermitian cusp forms of weight $k$. We will conclude by giving numerical examples for each case.  
\end{abstract}

\section{Introduction}
\label{intro}
As is well known, Ramanujan's congruence asserts that 
\begin{align*}
\sigma _{11}(n)\equiv \tau (n) \bmod{691},
\end{align*} 
where $\sigma _k(n)$ is the sum of $k$-th powers of the divisors of $n$ 
and $\tau (n)$ is the $n$-th Fourier coefficient of Ramanujan's 
$\Delta $ function. It should be noted that $\sigma _{11}(n)$ 
is the $n$-th Fourier coefficient of the elliptic Eisenstein series 
of weight $12$ for $n\ge 1$. We can regard this congruence in two ways:\\
~~~~~~~~~~~~~~~~~(I) Congruences between Hecke eigenvalues of 
two eigenforms.\\
~~~~~~~~~~~~~~~~~(II) Congruences between Fourier coefficients 
of two modular forms. \\
There are many studies on (I) related to congruence primes
in the case of modular forms with several variables. 
For example, see \cite{Br,Kat1,Kat2,Kat-Miz,Kuro1,Kuro2,Kuro3}. 
On the other hand, 
it seems that studies on (II) are little known. However, 
Bringmann-Heim \cite{Bri-Hei} 
studied some congruences between Fourier coefficients of two Jacobi forms. 
These properties essentially give congruences for Fourier coefficients of 
two Siegel modular forms of degree $2$. 

In the present paper, we study congruences for Siegel modular forms and for 
Hermitian modular 
forms in the sense of (II) in a way different from that of Bringmann-Heim. 
In particular, 
the results herein for the case of Hermitian modular forms are completely new. 
In addition, we should emphasize that there is a relation between 
the existence of a prime dividing the $k-1$-th generalized Bernoulli 
number and the existence of non-trivial Hermitian cusp forms of weight 
$k$ (Theorem \ref{Thm2}). We will conclude by giving numerical examples for each case.   
 
%%%%%%%%%%%%%%%%%%%%%%%%%%%%%%%%%%%%%%%%%%%%%%%%%%%%%%%%%%%%%%%%      
\section{Statement of results}
\subsection{The case of Siegel modular forms}
Let $M_k(\Gamma_n)$ denote the space of Siegel modular forms of weight 
$k$ for the Siegel modular group $\Gamma_n:=Sp_n(\mathbb{Z})$
and let $S_k(\Gamma_n)$ be the subspace of cusp forms. Any Siegel modular form
$F(Z)$ in $M_k(\Gamma_n)$ has a Fourier expansion of the form
\[
F(Z)=\sum_{0\leq T\in\Lambda_n}a_F(T)q^T,\quad q^T:=e^{2\pi i\text{tr}(TZ)},
\quad Z\in\mathbb{S}_n,
\]
where
\[
\Lambda_n
:=\{ T=(t_{ij})\in Sym_n(\mathbb{Q})\;|\; t_{ii},\;2t_{ij}\in\mathbb{Z}\; \}
\]
(the lattice in $Sym_n(\mathbb{R})$ of half-integral, symmetric matrices),
and $\mathbb{S}_n$ is the Siegel upper-half space of degree $n$.

For any subring $R\subset\mathbb{C}$, we define
\begin{align*}
& M_k(\Gamma_n)_R:=\{ F=\sum_{T\in\Lambda_n}a_F(T)q^T\;|\;
a_F(T)\in R\;(\forall\, T\in\Lambda_n)\;\},\\
& S_k(\Gamma_n)_R:=M_k(\Gamma_n)_R\cap S_k(\Gamma_n).
\end{align*}

In this paper, we consider the degree 2 case. A typical example of Siegel 
modular forms is the Siegel-Eisenstein series
\[
E_k(Z):=\sum_{M= \left( \begin{smallmatrix} * & * \\ C & D\end{smallmatrix} 
\right)}\det (CZ+D)^{-k},\quad Z\in \mathbb{S}_2,
\]
where $k>3$ is even, and $M$ runs over a set of representatives
$\left\{\left( \begin{smallmatrix} * & * \\ 0_2 & * \end{smallmatrix} 
\right)\right\}\backslash\Gamma_2$. 
It is known that
$E_k\in M_k(\Gamma_2)_{\mathbb{Q}}$. 
We set
\[
G_k:=-\frac{B_{k}B_{2k-2}}{4k(k-1)}E_k,
\]
where $B_m$ is the $m$-th Bernoulli number.

The first main result is as follows:
\begin{Thm}
\label{ThmM}
Let $(p,2k-2)$ be an irregular pair. Then there exists a cusp form 
$f\in M_k(\Gamma _2)_{\mathbb{Z}_{(p)}}$ such that 
\[
G_k\equiv f \bmod{p}.
\]   
\end{Thm}
\begin{Rem}
Let $p$ be a prime number. A pair $(p,m)$ is called $irregular$
if $1<m<p$ and $p|B_{m}$.\\
\end{Rem}

%%%%%%%%%%%%%%%%%%%%%%%%%%%%%%%%%%%%%%%%%%%%%%%%%%%%%%%%%%%%%%%%
\subsection{The case of Hermitian modular forms}
The Hermitian upper half-space of degree $n$ is defined by
\[
\mathbb{ H}_n:=\{ Z\in M_n(\mathbb{C})\;|\; 
\tfrac{1}{2i}(Z-{}^t\overline{Z})>0\;  \},
\]
where ${}^t\overline{Z}$ is the transpose, complex conjugate of $Z$. 
The Siegel upper-half space $\mathbb{S}_n$ is a subdomain of $\mathbb{H}_n$.
Let $\boldsymbol{K}$ be an imaginary quadratic number field with discriminant 
$d_{\boldsymbol{K}}$ and ring of integers $\mathcal{O}_{\boldsymbol{K}}$. 
The Hermitian modular group $U_n(\mathcal{O}_{\boldsymbol{K}})$ is defined by
\[
U_n(\mathcal{O}_{\boldsymbol{K}}):=
\{ M\in M_n(\mathcal{O}_{\boldsymbol{K}})\;|\; {}^t\overline{M}J_n M=J_n\},\; 
J_n=\begin{pmatrix}0_n & -1_n \\ 1_n & 0_n\end{pmatrix} 
\]
%%%%%
Define a character $\nu _k$ on $U_n(\mathcal{O}_{\boldsymbol{K}})$ as follows:
\[
\nu_k:=
\begin{cases}
\text{det}^{\frac{k}{2}} & \text{if $d_{\boldsymbol{K}}=-4$},\\
\text{det}^{k} & \text{if $d_{\boldsymbol{K}}=-3$},\\
\text{trivial character} & \text{otherwise}.
\end{cases}
\]
We denote by $M_k(U_n(\mathcal{O}_{\boldsymbol{K}}), \nu _k)$ the space of 
Hermitian modular forms of weight $k$ and character $\nu_k$ with respect to 
$U_n(\mathcal{O}_{\boldsymbol{K}})$. Namely, it consists of holomorphic 
functions $F:\mathbb{H}_n\longrightarrow \mathbb{C}$ satisfying
\[
F\mid_kM(Z):=\text{det}(CZ+D)^{-k}F((AZ+B)(CZ+D)^{-1})=\nu _k(M)F(Z),
\]
for all $M=\left( \begin{smallmatrix}A &B \\ C & D \end{smallmatrix} 
\right) \in U_n({\mathcal O}_{\boldsymbol K})$.

A cusp form is characterized by the condition
\[
\Phi (F\mid_k\binom{\,\!{}^t\overline{U}\;\;0_n}{\;0_n\,\;U}) 
\equiv 0\quad \text{for}\;\text{all}\;
U\in GL_n(\boldsymbol{K})
\]
where $\Phi$ is the Siegel $\Phi$-operator. If the class number of ${\boldsymbol K}$ 
is $1$, this condition is equivalent to $\Phi (F)\equiv 0$. We denote by 
$S_k(U_n(\mathcal{O}_{\boldsymbol{K}}),\nu_k)$ the subspace consisting of cusp forms.

If $F\in M_k(U_n(\mathcal{O}_{\boldsymbol{K}}),\nu_k)$, then $F$ has a Fourier 
expansion of the form
\[
F(Z)=\sum_{0\leq H\in\Lambda_n(\boldsymbol{K})}a_F(H)q^H,\quad 
q^H:=e^{2\pi i\text{tr}(HZ)},\quad Z\in\mathbb{H}_n,
\]
where
\[
\Lambda_n(\boldsymbol{K}):=\{ H=(h_{ij})\in Her_n(\boldsymbol{K})
\;|\;h_{ii}\in\mathbb{Z},\sqrt{d_{\boldsymbol{K}}}h_{ij}
\in\mathcal{O}_{\boldsymbol{K}} \}.
\]
Let $R$ be a subring of $\mathbb{C}$. As in the case of Siegel modular forms, 
we define
\begin{align*}
&M_k(U_n(\mathcal{O}_{\boldsymbol{K}}),\nu _k)_R\\
&:=\{ F=\sum a_F(H)q^H\in M_k(U_n(\mathcal{O}_{\boldsymbol{K}}),\nu _k)\,|
\, a_F(H)\in R\ (\forall H\in \Lambda_n(\boldsymbol{K}))\}, \\
& S_k(U_n (\mathcal{O}_{\boldsymbol{K}}),\nu _k)_R:=
M_k(U_n(\mathcal{O}_{\boldsymbol{K}}),\nu _k)_R\cap 
S_k(U_n (\mathcal{O}_{\boldsymbol{K}}),\nu _k).
\end{align*}

We consider the Hermitian Eisenstein series of degree 2
\[
E_{k,\boldsymbol K}(Z)
:=\sum_{M=\left( \begin{smallmatrix} * & * \\ C & D \end{smallmatrix} \right)}
({\det}M)^{\frac{k}{2}}{\det}(CZ+D)^{-k},\quad 
Z\in\mathbb{H}_2,
\]
where $k>4$ is even and $M$ runs over a set of representatives of 
$\left\{ \left( \begin{smallmatrix} * & * \\ 0_2 & *\end{smallmatrix} \right) \right\}
\backslash U_2(\mathcal{O}_{\boldsymbol{K}})$. 
Then we have
\[
E_{k,\boldsymbol K}\in M_k(U_2(\mathcal{O}_{\boldsymbol{K}}),\nu _k)_{\mathbb{Q}}.
\]
Moreover, $E_{4,\boldsymbol K}\in M_4(U_2(\mathcal{O}_{\boldsymbol{K}}),\nu _4)$ 
is constructed by the Maass lift (\cite{Kri}). 

We set 
\[
G_{k,{\boldsymbol K}}
:=\frac{B_k\cdot B_{k-1,\chi _{\boldsymbol K}}}{4k(k-1)}E_{k,{\boldsymbol K}},
\]
where $B_{m,\chi_{\boldsymbol{K}}}$ is the generalized Bernoulli number 
with respect to the
Kronecker character $\chi_{\boldsymbol{K}}$ of $\boldsymbol{K}$. 

In this paper, we treat only the case that
\[
\text{the class number of $\boldsymbol{K}$ is one}. 
\]
\begin{Rem}
It is well known that, if the class number of $\boldsymbol{K}$ is one,
then $d_{\boldsymbol{K}}$ is one of
\[
-3,-4,-7,-8,-11,-19,-43,-67,-163.
\]
\end{Rem}
Our second main theorem is 
\begin{Thm}
\label{Thm2}
Assume that the class number of $\boldsymbol{K}$ is one and $k$ is even.
Suppose that a prime number $p$ satisfies the following conditions (A) and (B).
\vspace{2mm}
\\
~~~~~~~~~~~~~~~~(A) $p\nmid B_{3,\chi _{\boldsymbol K}}$ and 
$p\nmid B_{5,\chi _{\boldsymbol K}}$, 
\vspace{2mm}
\\
~~~~~~~~~~~~~~~~(B) $k<p-1$ and $p|B_{k-1,\chi _{{\boldsymbol K}}}$, 
\vspace{2mm}
\\
Then there exists a $non$-$trivial$ cusp form 
$f\in S_{k}(U_2({\mathcal O}),\nu _k)_{\mathbb{Z}_{(p)}}$ such that 
\[
G_{k,{\boldsymbol K}}\equiv f \bmod{p}.
\] 
\end{Thm}
\begin{Cor}
\label{Cor1}
(1)\; Assume that $d_{\boldsymbol{K}}$ is one of $-3,-4,-7$. If $p$ 
satisfies condition (B), then 
there exists a $non$-$trivial$ cusp form 
$f\in S_{k}(U_2({\mathcal O}_{\boldsymbol K}),\nu _k)_{\mathbb{Z}_{(p)}}$ such that 
\[
G_{k,{\boldsymbol K}}\equiv f \bmod{p}.
\] 
(2j\; If $d_{\boldsymbol K}\neq -3$, then 
${\rm dim}S_{8}(U_2({\mathcal O}_{\boldsymbol K}),\nu _8)\ge 1$.  
\end{Cor}
\begin{Rem}
(1)\; It is known that there are non-trivial cusp forms of weight $10$ and $12$. \\
(2)\; As introduced by Hao-Parry \cite{Ha-Pa,Ha-Pa2}, the prime $p$ satisfying condition 
(B) is related to the divisibility by $p$ of the class number of a certain algebraic number field.
\end{Rem}
%%%%%%%%%%%%%%%%%%%%%%%%%%%%%%%%%%%%%%%%%%%%%%%%%%%
\section{Proof of theorems}
\subsection{Proof of Theorem \ref{ThmM}}
We start with proving that
\begin{Lem} 
\label{Lem}
Assume that $g=\sum a_g(T)q^T\in M_k(\Gamma _2)_{\mathbb{Z}_{(p)}}$ 
satisfies that $a_g(T)\equiv 0$ mod $p$ for all $T$ with $\det T=0$. 
Then there exists a cusp form $f\in S_k(\Gamma _2)_{\mathbb{Z}_{(p)}}$ 
such that
\[
g \equiv f \bmod{p}.
\]
\end{Lem}
\begin{proof}
Since the proof is same as in the case of Hermitian modular forms, 
we omit it. See the proof of Lemma \ref{Lem3} in $\S$ \ref{prf2}. 
\end{proof}

%%%%%%%%%%%%%%%%%
\begin{proof}[Proof of Theorem \ref{ThmM}]
We recall that the Fourier 
coefficients of the Eisenstein series $G_k$ are given as follows (e.g. cf. \cite{Ei-Za}):
\[
a_{G_k}(T)=
\begin{cases} 
\displaystyle 
-\frac{B_{k}\cdot B_{2k-2}}{4k(k-1)}\ & \text{if $T=0_2$},
\vspace{2mm}
\\ 
\displaystyle 
\frac{B_{2k-2}}{2k-2}\sigma _{k-1}(\varepsilon (T))\ & \text{if ${\rm rank}(T)=1$},
\vspace{2mm}
\\ 
\displaystyle 
\frac{B_{k-1,\chi _{D(T)}}}{k-1}
\sum _{d|\varepsilon (T)}d^{k-1}
\sum _{f|\frac{f(T)}{d}}\mu (f)\chi _{D(T)}(f)f^{k-2}\sigma _{2k-3}
\left(\frac{f(T)}{fd}\right)\ & \text{if\ ${\rm rank}(T)=2$},
\end{cases}
\]
where
\begin{align*}
& \text{$\varepsilon (T)$\,:\, the content of $T$, i.e. max$\{ l\in\mathbb{N}\,|\,l^{-1}T\in\Lambda_2\}$},\\
& \text{$\mu$\,:\, the M\"{o}bius function},\\
& \text{$D(T)$\,:\, the discriminant of the quadratic field $\mathbb{Q}(\sqrt{-4\det (T)})$},\\
& \text{$f(T)$\,:\, the natural number determined by $-4\det (T)=D(T)\cdot f(T)^2$},\\
& \text{$\chi_{D(T)}=\left(\frac{D(T)}{*} \right)$\,:\, the Kronecker character of $\mathbb{Q}(\sqrt{-4\det (T)})$},\\
& \text{$\sigma_m(N):=\sum_{0<d|N}d^m$}.
\end{align*}

We shall show that Fourier coefficients corresponding $T$ with $\text{rank}(T)\leq 1$
vanish modulo $p$. 
Note that $ p|\frac{B_{2k-2}}{2k-2}$ (by assumption),
and $\frac{B_{k}}{2k}\in \mathbb{Z}_{(p)}$ (by $k<p-1$
and von Staudt-Clausen's theorem). 
Hence, $a_{G_k}(0_2)\equiv 0 \bmod{p}$ and $a_{G_k}(T)\equiv 0 \bmod{p}$
for $T$ with $\text{rank}(T)=1$.

In order to complete the proof, we need to show that 
$a_{G_k}(T)\in \mathbb{Z}_{(p)}$ for all $T$ with ${\rm rank}(T)=2$. 
By a result of Maass \cite{Ma} (cf. also B\"ocherer \cite{Bo}), $a_{G_k}(T)$ 
can be written in the form 
\[
a_{G_k}(T)=\frac{1}{2D^{**}_{2k-2}}c(T),\quad\ \text{for some $c(T)\in \mathbb{Z}$},
\]
where $D^{**}_{2k-2}$ is the denominator of $B_{2k-2}$.
By the irregularity of $p$, the prime $p$ does not divide
$D^{**}_{2k-2}$. 
Therefore, we can take $g=G_k$ in Lemma \ref{Lem}.
This completes the proof of Theorem \ref{ThmM}.    
\end{proof}
\begin{Rem}
In  Theorem \ref{ThmM}, we do not refer to the non-triviality of the cusp form $f$
(compare with Theorem \ref{Thm2}). It is plausible that the non-triviality holds
for any prime $p$ satisfying the assumption of Theorem \ref{ThmM}.
Here we give a sufficient condition (see below $(*)$) for the non-triviality.

Under the same assumption as in Theorem \ref{ThmM}, we assume that
there exists a fundamental discriminant $D_0<0$ such that
\[
(*)\qquad\quad  p\nmid B_{k-1,\chi_{D_0}},
\]
then $f$ is non-trivial modulo $p$.

In \cite{Bru}, Bruinier proved that, if $p\nmid \frac{B_{2k-2}}{2k-2}$, then
there exists such $D_0$. However we cannot apply his result because
$p$ satisfies $p| \frac{B_{2k-2}}{2k-2}$ in our case.
\end{Rem}
%%%%%%%%%%%%%%%%%%%%%%%%%%%%%%%%%%%%%%%%%%%%%
\subsection{Proofs of Theorem \ref{Thm2} and Corollary \ref{Cor1}}
\label{prf2}
In this subsection, we prove Theorem \ref{Thm2} and Corollary \ref{Cor1}. 
Basically, the arguments are same as in the case of Siegel modular forms. 
First we prove 
\begin{Lem}
\label{Lem5}
Let $p$ be a prime satisfying condition (A) in Theorem \ref{Thm2}, 
namely, $p\nmid B_{3,\chi _{\boldsymbol K}}$ and 
$p\nmid B_{5,\chi _{\boldsymbol K}}$. 
Then we have
\[
E_{4,\boldsymbol K}\in 
M_{4}(U_2({\mathcal O}_{\boldsymbol K}),\nu _4)_{\mathbb{Z}_{(p)}}\;\; \text{and}\;\; 
E_{6,{\boldsymbol K}}\in 
M_{6}(U_2({\mathcal O}_{\boldsymbol K}),\nu _6)_{\mathbb{Z}_{(p)}}.
\]
\end{Lem}
\begin{proof}[Proof of Lemma \ref{Lem5}]
The formula for the Fourier coefficient of $E_{k,\boldsymbol K}$ is given by 
Krieg \cite{Kri} and Dern \cite{Der,Der2} as follows:  
\begin{equation}
\label{Coeff}
a_{E_{k,{\boldsymbol K}}}(H)=
\begin{cases} 
\displaystyle 
1 \ & \text{if $H=0_2$},
\vspace{2mm}
\\ 
\displaystyle 
-\frac{2k}{B_k}\sigma _{k-1}(\varepsilon (H))\ & \text{if ${\rm rank}(H)=1$},
\vspace{2mm}
\\ 
\displaystyle 
\frac{4k(k-1)}{B_k\cdot B_{k-1,\chi _{\boldsymbol K}}}
\sum _{d|\varepsilon (H)}d^{k-1} G_{\boldsymbol K}
\!\left(k-2,\frac{|d_{\boldsymbol K}|\det (H)}{d^2}\right)
& \text{if\ ${\rm rank}(H)=2$},
\end{cases}
\end{equation}
where $\varepsilon (H)$ is the content of $H$, $\sigma _{m}(N):=\sum _{d|N}d^m$ and
\begin{align*}
& G_{\boldsymbol K}
\!(m,N):=\frac{1}{1+|\chi_{\boldsymbol{K}}(N)|}
(\sigma _{m,\chi_{\boldsymbol{K}}}(N)-\sigma ^*_{m,\chi_{\boldsymbol{K}}}(N)),
\vspace{2mm}
\\
&\sigma _{m,\chi_{\boldsymbol{K}}}(N):=\sum _{0<d|N}\chi_{\boldsymbol{K}}(d)d^{m},
\quad  \sigma ^*_{m,\chi_{\boldsymbol{K}}}(N):=\sum _{0<d|N}\chi_{\boldsymbol{K}}
(N/d)d^{m}.
\end{align*}
Note that $G_{\boldsymbol K}
\!(m,N)\in \mathbb{Z}$. If $k=4$ or $6$, then 
$\frac{2k}{B_k}\in \mathbb{Z}$. 
Moreover by the assumption, we have 
$\frac{1}{B_{3,\chi _{\boldsymbol K}}}$, 
$\frac{1}{B_{5,\chi _{\boldsymbol K}}}\in \mathbb{Z}_{(p)}$. 
Hence, the claim follows. 
\end{proof}
\begin{Lem} 
\label{Lem3}
Let $p$ be a prime satisfying condition (A) in Theorem \ref{Thm2}, 
namely, $p\nmid B_{3,\chi _{\boldsymbol K}}$ and $p\nmid B_{5,\chi _{\boldsymbol K}}$. 
Assume that 
$g=\sum a_g(H)q^H\in M_k(U_2({\mathcal O}_{\boldsymbol K}),\nu _k)_{\mathbb{Z}_{(p)}}$ 
satisfies that $a_g(H)\equiv 0 \bmod{p}$ for all $H$ with $\det H=0$. 
Then there exists a cusp form 
$f\in S_k(U_2({\mathcal O}_{\boldsymbol K}),\nu _k)_{\mathbb{Z}_{(p)}}$ such that 
$g \equiv f \bmod{p}$. 
\end{Lem}
\begin{proof}[Proof of Lemma \ref{Lem3}]
Take a polynomial $Q(X,Y)\in \mathbb{Z}_{(p)}[X,Y]$ such that 
$\Phi (g)=Q(E^{(1)}_4,E^{(1)}_6)$, where $E_4^{(1)}$ and $E_6^{(1)}$ are 
the normalized elliptic Eisenstein series of weight $4$ and $6$, respectively. 
Then the polynomial $Q$ must satisfy $\widetilde{Q}=0$ in $\mathbb{F}_p[X,Y]$ 
by the result of Swinnerton-Dyer \cite{Swi}. Now we consider 
$f:=g-Q(E_{4,\boldsymbol K},E_{6,{\boldsymbol K}})$. We see that 
$f\in S_k(U_2({\mathcal O}_{\boldsymbol K}),\nu _k)$ because of $\Phi (f)=0$. 
Moreover, by Lemma \ref{Lem5}, 
$E_{4,\boldsymbol K}\in M_{4}(U_2({\mathcal O}_{\boldsymbol K}),\nu _4)_{\mathbb{Z}_{(p)}}$ 
and 
$E_{6,{\boldsymbol K}}\in M_{6}(U_2({\mathcal O}_{\boldsymbol K}),\nu _6)_{\mathbb{Z}_{(p)}}$. 
Therefore, $Q(E_{4,\boldsymbol K},E_{6,{\boldsymbol K}})\equiv 0 \bmod{p}$, and 
hence $g \equiv f \bmod{p}$.     
\end{proof}
 
We return to the proof of Theorem \ref{Thm2}. 
\begin{proof}[Proof of Theorem \ref{Thm2}]
The Fourier coefficient of $G_{k,\boldsymbol K}$ is given as follows:
\[
a_{G_{k,{\boldsymbol K}}}(H)=
\begin{cases} 
\displaystyle\frac{B_k\cdot B_{k-1,\chi _{\boldsymbol K}}}{4k(k-1)} & \text{if $H=0_2$},
\vspace{2mm}
\\ 
\displaystyle-\frac{B_{k-1,\chi _{\boldsymbol K}}}{2k-2}\sigma _{k-1}(\varepsilon (H)) &
\text{if rank$(H)=1$},
\vspace{2mm}
\\ 
\displaystyle\sum _{d|\varepsilon (H)}d^{k-1} G_{\boldsymbol K}\!
\left(k-2,\frac{|d_{\boldsymbol K}|\det (H)}{d^2}\right)
& \text{if rank$(H)=2$},
\end{cases}
\]
where the notation is as in (\ref{Coeff}).

From the assumption $k<p-1$ and von Staudt-Clausen's theorem, we have
$\frac{B_k}{2k}\in \mathbb{Z}_{(p)}$.
Hence, $p\,|\,\frac{B_k\cdot B_{k-1,\chi _{\boldsymbol K}}}{4k(k-1)}$ 
and $p\,|\,\frac{B_{k-1,\chi _{\boldsymbol K}}}{2k-2}$. 
These facts imply that $a_{G_{k,{\boldsymbol K}}}(0_2)\equiv 0 \bmod{p}$ for all $H$
and $a_{G_{k,{\boldsymbol K}}}(H)\equiv 0 \bmod{p}$ 
with ${\rm rank}(H)=1$. For $H$ with ${\rm rank }(H)=2$, 
$a_{G_{k,\boldsymbol K}}(H)\in \mathbb{Z}$. 
Namely, 
$G_{k,{\boldsymbol K}}\in 
M_{k}(U_2({\mathcal O}_{{\boldsymbol K}}),\nu _k)_{\mathbb{Z}_{(p)}}$ 
and this implies that $a_{G_{k,\boldsymbol K}}(H)\equiv 0 \bmod{p}$ 
for all $H$ with $\det H=0$. 
Applying Lemma \ref{Lem3} to $G_{k,\boldsymbol K}$, there exists a cusp form such that 
$G_{k,\boldsymbol K}\equiv f \bmod{p}$.

In order to complete the proof of Theorem \ref{Thm2}, we need to show the non-triviality of $f$. 
The proof is reduced to proving that there exists $H$ with ${\rm rank}(H)=2$ which 
satisfies $a_{G_{k,{\boldsymbol K}}}(H)\not\equiv 0 \bmod{p}$. First, 
we consider the Fourier coefficient corresponding $H=1_2$ (the unit matrix of degree 2), 
which implies 
$a_{G_{k,{\boldsymbol K}}}(1_2)=G_{\boldsymbol K}(k-2,|d_{\boldsymbol K}|)=1-|d_{\boldsymbol K}|^{k-2}$. 
If $|d_{\boldsymbol K}|^{k-2}\not \equiv 1 \bmod{p}$, then 
$a_{G_{k,{\boldsymbol K}}}(1_2)\not \equiv 0 \bmod{p}$. This shows the non-triviality of
$G_{k,{\boldsymbol K}}$ in this case.
Hence, it suffices to consider the case that $|d_{\boldsymbol K}|^{k-2}\equiv 1 \bmod{p}$. 
Note that $(p,|d_{\boldsymbol K}|)=1$ in this case. Now we prove
%%%%%%%%%%%%%%%
\begin{Lem}
\label{Lem2}
Assume that $k-2<p-1$ and the prime $p$ does not divide $d_{\boldsymbol K}$. 
Then there exists a prime $q$ such that 
$\chi _{\boldsymbol K}(q)=-1$ and $q^{k-2}\not \equiv 1 \bmod{p}$. 
\end{Lem}
%%%%%%%%%%%%%%%%%%%%%%
\begin{proof}[Proof of Lemma \ref{Lem2}]
Let $\alpha $ be an integer such that 
$(\mathbb{Z}/p\mathbb{Z})^{\times }=\left< \overline{\alpha} \right>$. 
Applying the Chinese remainder theorem, we can find an integer $a$ such that 
$a\equiv \alpha \bmod{p}$ and $a\equiv -1 \bmod{|d_{\boldsymbol K}|}$. 
It is clear that $(a,p\cdot |d_{\boldsymbol K}|)=1$ for such $a$. 
By Dirichlet's theorem on arithmetic progressions, there exist infinitely 
many primes in the sequence $\{a+p\cdot |d_{\boldsymbol K}|\cdot n\}_{n=1}^\infty$. 
We take a prime $q$ appearing in this sequence. 
Then $q$ satisfies both $\chi _{\boldsymbol K}(q)=-1$ and 
$q^{k-2}\not \equiv 1 \bmod{p}$.
In fact, $q\equiv a \equiv \alpha \bmod{p}$, and hence $q^{k-2}\not\equiv 1 \bmod{p}$ 
because $k-2<p-1$. 
On the other hand, by $q\equiv a\equiv -1$ mod $|d_{\boldsymbol K}|$, we have 
$\chi _{\boldsymbol K}(q)=-1$.  
\end{proof}

We return to the proof of the non-triviality of $f$.
Note that $k-2<p-1$ by the assumption in Theorem \ref{Thm2}. 
By Lemma \ref{Lem2}, there exists a prime $q$ such that $\chi _{{\boldsymbol K}}(q)=-1$ 
and $q^{k-2}\not \equiv 1 \bmod{p}$. 
Considering the case $H=\left( \begin{smallmatrix} 1 & 0 \\ 0 & q \end{smallmatrix}\right)$, 
we have 
\begin{align*}
a_{G_{k,{\boldsymbol K}}}(H)&=G_{\boldsymbol K}\!
(k-2,|d_{\boldsymbol K}|q)\\
&=1-q^{k-2}+|d_{\boldsymbol K}|^{k-2}-|d_{\boldsymbol K}|^{k-2}q^{k-2}\\
&=(1-q^{k-2})(1+|d_{\boldsymbol K}|^{k-2})\\
&\equiv 2(1-q^{k-2}) \bmod{p}. 
\end{align*} 
By the choice of $q$, this is not $0$ modulo $p$. 
This completes the proof of Theorem \ref{Thm2}. 
\end{proof}
%%%%%%%%%%%%%%
\begin{proof}[Proof of Corollary \ref{Cor1}]
See tables in $\S$ \ref{Table}. \\
(1)\; For the cases $d_{\boldsymbol{K}}=-3,-4,-7$, condition (A) is always satisfied.
\\
(2)\; If $d_{\boldsymbol K}\neq -3$ and $k=8$, then there exists a prime $p$ satisfying conditions (A) and (B). 
Thus we the assertion of Corollary \ref{Cor1} (2) is obtained. 
\end{proof}

%%%%%%%%%%%%%%%%%%%%%%%%%%%%%%%%%%%%%%%%%%%%%%%
\section{Examples}
\subsection{The case of Siegel modular forms}
We set
\begin{equation*}
\begin{split}
X_{10}:&=-\frac{43867}{2^{10}\cdot 3^5\cdot 5^2\cdot 7\cdot 53}(E_{10}-E_4E_6), \\
X_{12}:&=\frac{131\cdot 593}{2^{11}\cdot 3^6\cdot 5^3\cdot 7^2\cdot 337}
(3^2\cdot 7^2E_4^3+2\cdot 5^3E_6^2-691E_{12}).
\end{split}
\end{equation*}
The following are cusp forms defined by Igusa. 
%%%%%%%%%%%%%%%%%%%%%%%%%%%%%%%%%%%%%%%%%
\subsubsection*{The case of weight 10}
\[
G_{10} \equiv 11313\cdot X_{10} \bmod{43867}.
\]
\begin{align*}
 & a_{G_{10}}\!\!\left(\begin{pmatrix}
1 & \tfrac{1}{2}\\
\tfrac{1}{2} & 1
\end{pmatrix}  
\right)
=-\frac{1618}{27}  \equiv  11313=
11313\cdot
a_{X_{10}}\!\!\left(\begin{pmatrix}
1 & \tfrac{1}{2}\\
\tfrac{1}{2} & 1
\end{pmatrix}  
\right)
\bmod 43867,\\
%%%%%%%%%
 & a_{G_{10}}\!\!\left(\begin{pmatrix}
1 & 0\\
0 & 1
\end{pmatrix}  
\right)
=-\frac{1385}{2}  \equiv  11313\cdot (-2)=
11313\cdot
a_{X_{10}}\!\!\left(\begin{pmatrix}
1 & 0\\
0 & 1
\end{pmatrix}  
\right)
\bmod 43867,\\
%%%%%%%%
 & a_{G_{10}}\!\!\left(\begin{pmatrix}
1 & \tfrac{1}{2}\\
\tfrac{1}{2} & 2
\end{pmatrix}  
\right)
=-\frac{565184}{7}  \equiv  11313\cdot (-16)=
11313\cdot
a_{X_{10}}\!\!\left(\begin{pmatrix}
1 & \tfrac{1}{2}\\
\tfrac{1}{2} & 2
\end{pmatrix}    
\right)
\bmod 43867,\\
%%%%%%%%%
 & a_{G_{10}}\!\!\left(\begin{pmatrix}
1 & 0\\
0 & 2
\end{pmatrix}  
\right)
=-250737  \equiv  11313\cdot 36=
11313\cdot
a_{X_{10}}\!\!\left(\begin{pmatrix}
1 & 0\\
0 & 2
\end{pmatrix}  
\right)
\bmod 43867.
\end{align*}
%%%%%%%%%%%%%%%%%%%%%%%%%%%%%%%%%%%%%%%%%%%%%%%%
\subsubsection*{The case of weight 12}
\[
G_{12} \equiv 53020\cdot X_{12} \bmod{131\cdot 593}.
\]
\begin{align*}
 & a_{G_{12}}\!\!\left(\begin{pmatrix}
1 & \tfrac{1}{2}\\
\tfrac{1}{2} & 1
\end{pmatrix}  
\right)
=\frac{3694}{3}  \equiv  53020=
53020\cdot
a_{X_{12}}\!\!\left(\begin{pmatrix}
1 & \tfrac{1}{2}\\
\tfrac{1}{2} & 1
\end{pmatrix}  
\right)
\bmod 131\cdot 593,\\
%%%%%%%%%
 & a_{G_{12}}\!\!\left(\begin{pmatrix}
1 & 0\\
0 & 1
\end{pmatrix}  
\right)
=\frac{5052}{2}  \equiv  53020\cdot 10=
53020\cdot
a_{X_{12}}\!\!\left(\begin{pmatrix}
1 & 0\\
0 & 1
\end{pmatrix}  
\right)
\bmod 131\cdot 593,\\
%%%%%%%%
 & a_{G_{12}}\!\!\left(\begin{pmatrix}
1 & \tfrac{1}{2}\\
\tfrac{1}{2} & 2
\end{pmatrix}  
\right)
=9006448  \equiv 53020\cdot (-88)=
53020\cdot
a_{X_{12}}\!\!\left(\begin{pmatrix}
1 & \tfrac{1}{2}\\
\tfrac{1}{2} & 2
\end{pmatrix}    
\right)
\bmod 131\cdot 593,\\
%%%%%%%%%
 & a_{G_{12}}\!\!\left(\begin{pmatrix}
1 & 0\\
0 & 2
\end{pmatrix}  
\right)
=36581523  \equiv  53020\cdot (-132)=
53020\cdot
a_{X_{12}}\!\!\left(\begin{pmatrix}
1 & 0\\
0 & 2
\end{pmatrix}  
\right)
\bmod 131\cdot 593.
\end{align*}
%%%%%%%%%%%%%%%%%%%%%%%%%%%%%%%%%%%%%%%%%%%%%%%%%%%%%%%%%%%%%%%%%%%%%%%
\subsection{The case of Hermitian modular forms}
In this subsection, we introduce congruences for
Hermitian modular forms. We use modular forms constructed
in \cite{Kiku-Nag}.
%%%%%%%%%%%%%%%%%%%%%%%%%%%%%%%%%%%%%%%%%
\subsubsection*{Examples for $d_{\boldsymbol K}=-3$}
We set
\begin{equation*}
\begin{split}
F_{10}&= F_{10,\mathbb{Q}(\sqrt{-3})}
:=-\frac{809}{21772800}
(E_{10,\mathbb{Q}(\sqrt{-3})}-E_{4,\mathbb{Q}(\sqrt{-3})}\cdot
E_{6,\mathbb{Q}(\sqrt{-3})}),\\
F_{12}&= F_{12,\mathbb{Q}(\sqrt{-3})}
:=-\frac{1276277}{36578304000}
\left(E_{12,\mathbb{Q}(\sqrt{-3})}-\frac{441}{691}E_{4,\mathbb{Q}(\sqrt{-3})}^3
-\frac{250}{691}
E_{6,\mathbb{Q}(\sqrt{-3})}^2\right)
\end{split}
\end{equation*}
(cf. \cite{Kiku-Nag}, Lemma 4.8 and Lemma 4.9).
%%%%%%%%%%%%%%%%%%%%%%%%%%%%%%%%%%%%%%%%%%%%%%%%%%%%%%%%%%%%%
\subsubsection*{The case of weight 10}
\[
G_{10} \equiv 554\cdot F_{10} \bmod{809},
\]
where $G_{10}=G_{10,\mathbb{Q}(\sqrt{-3})}$.
\begin{align*}
 & a_{G_{10}}\!\!\left(\begin{pmatrix}
1 & \tfrac{1}{\sqrt{3}i}\\
-\tfrac{1}{\sqrt{3}i} & 1
\end{pmatrix}  
\right)
=-255  \equiv  554=
554\cdot
a_{F_{10}}\!\!\left(\begin{pmatrix}
1 & \tfrac{1}{\sqrt{3}i}\\
-\tfrac{1}{\sqrt{3}i} & 1
\end{pmatrix}   
\right)
\bmod 809,\\
%%%
& a_{G_{10}}\!\!\left(\begin{pmatrix}
1 & 0\\
0 & 1
\end{pmatrix}  
\right)
=-6560  \equiv  554\cdot (-6)=
554\cdot
a_{F_{10}}\!\!\left(\begin{pmatrix}
1 & 0\\
0 & 1
\end{pmatrix}  
\right)
\bmod 809,\\
%%%%%%%%%
& a_{G_{10}}\!\!\left(\begin{pmatrix}
1 & \tfrac{1}{\sqrt{3}i}\\
-\tfrac{1}{\sqrt{3}i} & 2
\end{pmatrix}  
\right)
=-390624  \equiv  554\cdot (-10)=
554\cdot
a_{F_{10}}\!\!\left(\begin{pmatrix}
1 & \tfrac{1}{\sqrt{3}i}\\
-\tfrac{1}{\sqrt{3}i} & 2
\end{pmatrix}   
\right)
\bmod 809,\\
%%%
& a_{G_{10}}\!\!\left(\begin{pmatrix}
1 & 0\\
0 & 2
\end{pmatrix}  
\right)
=-1673310  \equiv  554\cdot 90=
554\cdot
a_{F_{10}}\!\!\left(\begin{pmatrix}
1 & 0\\
0 & 2
\end{pmatrix}  
\right)
\bmod 809.
\end{align*}
%%%%%%%%%%%%%%%%%%%%%%%%%%%%%%%%%%%%%%%%%%%%%%%
\subsubsection*{The case of weight 12}
\[
G_{12} \equiv 824\cdot F_{12} \bmod{1847},
\]
where $G_{12}=G_{12,\mathbb{Q}(\sqrt{-3})}$.
\begin{align*}
 & a_{G_{12}}\!\!\left(\begin{pmatrix}
1 & \tfrac{1}{\sqrt{3}i}\\
-\tfrac{1}{\sqrt{3}i} & 1
\end{pmatrix}  
\right)
=-1023  \equiv 824=
824\cdot
a_{F_{12}}\!\!\left(\begin{pmatrix}
1 & \tfrac{1}{\sqrt{3}i}\\
-\tfrac{1}{\sqrt{3}i} & 1
\end{pmatrix}   
\right)
\bmod 1847,\\
%%%
& a_{G_{12}}\!\!\left(\begin{pmatrix}
1 & 0\\
0 & 1
\end{pmatrix}  
\right)
=-59048  \equiv  824\cdot 18=
824\cdot
a_{F_{12}}\!\!\left(\begin{pmatrix}
1 & 0\\
0 & 1
\end{pmatrix}  
\right)
\bmod 1847,\\
%%%%%%%%%
& a_{G_{12}}\!\!\left(\begin{pmatrix}
1 & \tfrac{1}{\sqrt{3}i}\\
-\tfrac{1}{\sqrt{3}i} & 2
\end{pmatrix}  
\right)
=-9765624  \equiv  824\cdot (-106)=
824\cdot
a_{F_{12}}\!\!\left(\begin{pmatrix}
1 & \tfrac{1}{\sqrt{3}i}\\
-\tfrac{1}{\sqrt{3}i} & 2
\end{pmatrix}   
\right)
\bmod 1847,\\
%%%
& a_{G_{12}}\!\!\left(\begin{pmatrix}
1 & 0\\
0 & 2
\end{pmatrix}  
\right)
=-60408150  \equiv  824\cdot (-54)=
824\cdot
a_{F_{12}}\!\!\left(\begin{pmatrix}
1 & 0\\
0 & 2
\end{pmatrix}  
\right)
\bmod 1847.
\end{align*}
%%%%%%%%%%%%%%%%%%%%%%%%%%%%%%%%%%%%%%%%%%%
\subsubsection*{Examples for $d_{\boldsymbol K}=-4$}
We set
\begin{equation*}
\begin{split}
\chi_8&
:=-\frac{61}{230400}
(E_{8,\mathbb{Q}(\sqrt{-1})}-E_{4,\mathbb{Q}(\sqrt{-1})}^2),
\\
F_{10}&= F_{10,\mathbb{Q}(\sqrt{-1})}
:=-\frac{277}{2419200}
(E_{10,\mathbb{Q}(\sqrt{-1})}-E_{4,\mathbb{Q}(\sqrt{-1})}
E_{6,\mathbb{Q}(\sqrt{-1})})
\end{split}
\end{equation*}
(cf. \cite{Kiku-Nag}, Lemma 4.3 and Lemma 4.4).\\
%%%%%%%%%%%%%%%%%%%%%%%%%%%%%%%%%%%%%%%%%%%%%%%%%%%%%%%%%%%%%%%%
\subsubsection*{The case of weight 8}
\[
G_8 \equiv -2\cdot \chi_8 \bmod{61},
\]
where $G_8=G_{8,\mathbb{Q}(\sqrt{-1})}$.
\begin{align*}
 & a_{G_8}\!\!\left(\begin{pmatrix}
1 & \tfrac{1+i}{2}\\
\tfrac{1-i}{2} & 1
\end{pmatrix}  
\right)
=-63  \equiv  -2=
-2\cdot
a_{\chi_8}\!\!\left(\begin{pmatrix}
1 & \tfrac{1+i}{2}\\
\tfrac{1-i}{2} & 1
\end{pmatrix}   
\right)
\bmod 61,\\
%%%
& a_{G_8}\!\!\left(\begin{pmatrix}
1 & \tfrac{1}{2}\\
\tfrac{1}{2} & 1
\end{pmatrix}  
\right)
=-728  \equiv  -2\cdot (-2)=
-2\cdot
a_{\chi_8}\!\!\left(\begin{pmatrix}
1 & \tfrac{1}{2}\\
\tfrac{1}{2}& 1
\end{pmatrix}  
\right)
\bmod 61,\\
%%%%%%%%%
& a_{G_8}\!\!\left(\begin{pmatrix}
1 & 0\\
0 & 1
\end{pmatrix}  
\right)
=-4095  \equiv  -2\cdot 4=
-2\cdot
a_{\chi_8}\!\!\left(\begin{pmatrix}
1 & 0\\
0 & 1
\end{pmatrix}   
\right)
\bmod 61,\\
%%%
& a_{G_8}\!\!\left(\begin{pmatrix}
1 & \tfrac{1+i}{2}\\
\tfrac{1-i}{2} & 2
\end{pmatrix}  
\right)
=-47320  \equiv  -2\cdot (-8)=
-2\cdot
a_{\chi_8}\!\!\left(\begin{pmatrix}
1 & \tfrac{1+i}{2}\\
\tfrac{1-i}{2} & 2
\end{pmatrix}  
\right)
\bmod 61.
\end{align*}
%%%%%%%%%%%%%%%%%%%%%%%%
\subsubsection*{The case of weight 10}
\[
G_{10} \equiv 22\cdot F_{10} \bmod{277},
\]
where $G_{10}=G_{10,\mathbb{Q}(\sqrt{-1})}$.
%%%%%%%%%%%%%%%%%%%%%%
\begin{align*}
 & a_{G_{10}}\!\!\left(\begin{pmatrix}
1 & \tfrac{1+i}{2}\\
\tfrac{1-i}{2} & 1
\end{pmatrix}  
\right)
=-255  \equiv  22=
22\cdot
a_{F_{10}}\!\!\left(\begin{pmatrix}
1 & \tfrac{1+i}{2}\\
\tfrac{1-i}{2} & 1
\end{pmatrix}   
\right)
\bmod 277,\\
%%%
& a_{G_{10}}\!\!\left(\begin{pmatrix}
1 & \tfrac{1}{2}\\
\tfrac{1}{2} & 1
\end{pmatrix}  
\right)
=-6560  \equiv  22\cdot 4=
22\cdot
a_{F_{10}}\!\!\left(\begin{pmatrix}
1 & \tfrac{1}{2}\\
\tfrac{1}{2}& 1
\end{pmatrix}  
\right)
\bmod 277,\\
%%%%%%%%%
& a_{G_{10}}\!\!\left(\begin{pmatrix}
1 & 0\\
0 & 1
\end{pmatrix}  
\right)
=-65535  \equiv  22\cdot (-20)=
22\cdot
a_{F_{10}}\!\!\left(\begin{pmatrix}
1 & 0\\
0 & 1
\end{pmatrix}   
\right)
\bmod 277,\\
%%%
& a_{G_{10}}\!\!\left(\begin{pmatrix}
1 & \tfrac{1+i}{2}\\
\tfrac{1-i}{2} & 2
\end{pmatrix}  
\right)
=-1685920  \equiv  22\cdot (-80)=
22\cdot
a_{F_{10}}\!\!\left(\begin{pmatrix}
1 & \tfrac{1+i}{2}\\
\tfrac{1-i}{2} & 2
\end{pmatrix}  
\right)
\bmod 277.
\end{align*}
%%%%%%%%%%%%%%%
%%%%%%%%%%%%%%%%%%%%%%%%%%%%%%%%%%%%%%%%%%%%%%%%%%%%%%%%%%%%%%%%%%%%%%%%%%%%%%%%%
\section{Values of the generalized Bernoulli numbers}
\subsection*{Values of $B_{k-1,\chi _{\boldsymbol K}}=\frac{N}{D}$ and $p$ satisfying 
(B) in Theorem \ref{Thm2}}
\label{Table}

\begin{minipage}{0.5\hsize}
\begin{center}
$d_{\boldsymbol K}=-3$ \\
\begin{tabular}{c|ccc}
$k-1$ & $N$  & $D$ & $p$ \\ 
\hline
$1$ & $-1$ & $3$ & -\\ 
$3$ & $2$ & $3$ & -\\
$5$ & $-2\cdot 5$ & $3$ & -\\
$7$ & $2\cdot 7^2$ & $3$ & -\\
$9$ & $-2\cdot 809$ & $3$ & $809$ \\
$11$ & $2\cdot 11\cdot 1847$ & $3$ & $1847$ \\
$13$ & $-2\cdot 7\cdot 13^3 \cdot 47$ & $3$ & $47$\\
$15$ & $2\cdot 5\cdot 419 \cdot 16519$ & $3$ & $419,\ 16519$
\end{tabular}
\end{center}
\end{minipage}
\begin{minipage}{0.5\hsize}
\begin{center}
$d_{\boldsymbol K}=-4$\\
\begin{tabular}{c|ccc}
$k-1$ & $N$ & $D$ & $p$ \\ 
\hline
$1$ & $-1$ & $2$ & -\\ 
$3$ & $3$ & $2$ & -\\
$5$ & $-5^2$ & $2$ & -\\
$7$ & $7\cdot 61$ & $2$ & $61$ \\
$9$ & $-3^2\cdot 5\cdot 277$ & $2$ & $277$ \\
$11$ & $11\cdot 19 \cdot 2659$ & $2$ & $19,\ 2659$ \\
$13$ & $-5\cdot 13^2\cdot 43 \cdot 967$ & $2$ & $43,\ 967$\\
$15$ & $3\cdot 5\cdot 47 \cdot 4241723$ & $2$ & $47,\ 4241723$
\end{tabular}
\end{center}
\end{minipage}

\newpage

\begin{center}
$d_{\boldsymbol K}=-7$\\
\begin{tabular}{c|ccc}
$k-1$ & $N$ & $D$ & $p$ \\ 
\hline
$1$ & $-1$ & $1$ & -\\ 
$3$ & $2^4\cdot 3$ & $7$ & -\\
$5$ & $-2^5 \cdot 5$ & $1$ & -\\
$7$ & $2^4 \cdot 7 \cdot 73$ & $1$ & $73$\\
$9$ & $-2^6\cdot 3^2\cdot 8831$ & $7$ & $8831$ \\
$11$ & $2^4\cdot 11^2 \cdot 73\cdot 701$ & $1$ & $73,\ 701$ \\
$13$ & $-2^5\cdot 13\cdot 173 \cdot 266447$ & $1$ & $173,\ 266447$\\
$15$ & $2^4\cdot 3\cdot 5 \cdot 145764975331$ & $7$ & $145764975331$
\end{tabular}
\end{center}

\begin{center}
$d_{\boldsymbol K}=-8$\\
\begin{tabular}{c|ccc}
$k-1$ & $N$ & $D$ & $p$ \\ 
\hline
$1$ & $-1$ & $1$ & -\\ 
$3$ & $3^2$ & $1$ & -\\
$5$ & $-3 \cdot 5 \cdot 19$ & $1$ & $19$ \\
$7$ & $3^2 \cdot 7 \cdot 307$ & $1$ & $307$\\
$9$ & $-3^3\cdot 83579$ & $1$ & $83579$ \\
$11$ & $3\cdot 11^2\cdot 23 \cdot 48197$ & $1$ & $23,\ 48197$ \\
$13$ & $-3^2\cdot 13\cdot 113 \cdot 811\cdot 9491$ & $1$ & $113,\ 811,\ 9491$\\
$15$ & $3^2\cdot 5\cdot 83\cdot 9275681267$ & $1$ & $83,\ 9275681267$
\end{tabular}
\end{center}

\begin{center}
$d_{\boldsymbol K}=-11$\\
\begin{tabular}{c|ccc}
$k-1$ & $N$ & $D$ & $p$ \\ 
\hline
$1$ & $-1$ & $1$ & -\\ 
$3$ & $2\cdot 3^2$ & $1$ & -\\
$5$ & $- 2\cdot 3 \cdot 5^3 \cdot 17$ & $11$ & $17$ \\
$7$ & $2 \cdot 3^2 \cdot 7 \cdot 17 \cdot 71$ & $1$ & $17,\ 71$\\
$9$ & $-2\cdot 3^3 \cdot 5^3 \cdot 4999$ & $1$ & $4999$ \\
$11$ & $2\cdot 3\cdot 11 \cdot 43 \cdot 269\cdot 14923$ & $1$ & $43,\ 269,\ 14923$ \\
$13$ & $-2\cdot 3^2 \cdot 5^2 \cdot 13 \cdot 787 \cdot 1183579$ & $1$ & $787, \ 1183579$\\
$15$ & $2\cdot 3^2 \cdot 5\cdot 428708869630871$ & $11$ & $428708869630871$
\end{tabular}
\end{center}

\begin{center}
$d_{\boldsymbol K}=-19$\\
\begin{tabular}{c|ccc}
$k-1$ & $N$ & $D$ & $p$ \\ 
\hline
$1$ & $-1$ & $1$ & -\\ 
$3$ & $2\cdot 3 \cdot 11$ & $1$ & $11$\\
$5$ & $- 2\cdot 5^2 \cdot 269$ & $1$ & $269$ \\
$7$ & $2 \cdot 7^2 \cdot 53 \cdot 1021$ & $1$ & $53,\ 1021$\\
$9$ & $-2\cdot 3^2 \cdot 5 \cdot 13 \cdot 67\cdot 851537$ & $19$ & $13,\ 67,\ 851537$ \\
$11$ & $2\cdot 11^3 \cdot 41 \cdot 32427511$ & $1$ & $41,\ 32427511$ \\
$13$ & $-2\cdot 5 \cdot 7 \cdot 11 \cdot 13 \cdot 149 \cdot 3386245229$ & $1$ & $149,\ 3386245229$\\
$15$ & $2\cdot 3 \cdot 5 \cdot 829 \cdot 1249187 \cdot 312206737$ & $1$ & $829,\ 1249187,\ 312206737$
\end{tabular}
\end{center}

\begin{center}
$d_{\boldsymbol K}=-43$\\
\begin{tabular}{c|ccc}
$k-1$ & $N$ & $D$ & $p$ \\ 
\hline
$1$ & $-1$ & $1$ & -\\ 
$3$ & $2\cdot 3 \cdot 83$ & $1$ & $83$\\
$5$ & $- 2\cdot 5 \cdot 29\cdot 31 \cdot 59$ & $1$ & $29,\ 31,\ 59$ \\
$7$ & $2 \cdot 7 \cdot 76565663$ & $1$ & $76565663$\\
$9$ & $-2\cdot 3^2 \cdot 202075601281$ & $1$ & $202075601281$ \\
$11$ & $2\cdot 11^2 \cdot 13^2 \cdot 509\cdot 901553753$ & $1$ & $509,\ 901553753$ \\
$13$ & $-2\cdot 13^2 \cdot 405842695582800517$ & $1$ & $405842695582800517$\\
$15$ & $2\cdot 3 \cdot 5 \cdot 223 \cdot 2791 \cdot 25889\cdot 113167\cdot 24665497$ & $1$ & $223,\ 2791,\ 25889\, 113167\, 24665497$
\end{tabular}
\end{center}

\begin{center}
$d_{\boldsymbol K}=-67$\\
\begin{tabular}{c|ccc}
$k-1$ & $N$ & $D$ & $p$ \\ 
\hline
$1$ & $-1$ & $1$ & -\\ 
$3$ & $2\cdot 3 \cdot 251$ & $1$ & $251$\\
$5$ & $- 2\cdot 5 \cdot 19^2 \cdot 23 \cdot 47$ & $1$ & $19,\ 23,\ 47$ \\
$7$ & $2 \cdot 7 \cdot 1367650871$ & $1$ & $1367650871$\\
$9$ & $-2\cdot 3^2 \cdot 151 \cdot 58035119431$ & $1$ & $151,\ 58035119431$ \\
$11$ & $2\cdot 11 \cdot 3272681\cdot 27444275311$ & $1$ & $3272681,\ 27444275311$ \\
$13$ & $-2\cdot 13 \cdot 73\cdot 1439\cdot 56783\cdot 226088481721$ & $1$ & $73,\ 1439,\ 56783,\ 226088481721$\\
$15$ & $2\cdot 3 \cdot 5 \cdot 541355166251\cdot 51558395838661$ & $1$ & $541355166251,\ 51558395838661$
\end{tabular}
\end{center}

\begin{center}
$d_{\boldsymbol K}=-163$\\
\begin{tabular}{c|ccc}
$k-1$ & $N$ & $D$ & $p$ \\ 
\hline
$1$ & $-1$ & $1$ & -\\ 
$3$ & $2\cdot 3 \cdot 5\cdot 463$ & $1$ & $463$\\
$5$ & $- 2\cdot 5 \cdot 13^2 \cdot 281 \cdot 449$ & $1$ & $13,\ 281,\ 449$ \\
$7$ & $2 \cdot 5^3 \cdot 7 \cdot 3538330867$ & $1$ & $3538330867$\\
$9$ & $-2\cdot 3^2 \cdot 47 \cdot 1213\cdot 294217150811$ & $1$ & $47,\ 1213,\ 294217150811$ \\
$11$ & $2\cdot 5 \cdot 11 \cdot 29^2 \cdot 179\cdot 379\cdot 3566823499667$ & $1$ & $29,\ 179,\ 379,\ 3566823499667$ \\
$13$ & $-2\cdot 13 \cdot 103\cdot 172357\cdot 1097359\cdot 1883639\cdot 2464211$ & $1$ & $103,\ 172357,\ 1097359,\ 1883639,\ 2464211$\\
$15$ & $2\cdot 3 \cdot 5^2 \cdot 358181\cdot 6185071975972339006627199$ & $1$ & $ 358181,\ 6185071975972339006627199$
\end{tabular}
\end{center}

%%%%%%%%%%%%%%%%%%%%%%%%%%%%%%%%%%%%%%%%%%%%%%%%%%%%%%%%%%%%%%%%%%%%%%%%%%%%%%%%%%


\begin{thebibliography}{}

\bibitem {Bo} S.~B\"ocherer, \"Uber die Fourierkoeffizienten der Siegelschen Eisensteinreihen, manuscripta math. {\bf 45} (1985), 273-288.

\bibitem{Br} J.~Brown, On the congruences primes of Saito-Kurokawa lifts, Math. Res. Lett. 
\textbf{17}(05) (2010), 977-991.

\bibitem{Bri-Hei} K.~Bringmann, B.~Heim, Ramanujan type congruences and Jacobi forms, International Mathematics Research Notices, 2007. 

\bibitem{Bru} J.~H.~Bruinier, Non-vanishing modulo $l$ of Fourier coefficients of half-integral weight modular forms, Duke Math. J. \textbf{98}(3) (1999), 595-611. 

\bibitem{Der} T.~Dern, Hermitesche Modulformen zweiten Grades, Verlag Mainz, Wissenschaftsverlag, Aachen, 2001.

\bibitem{Der2} T.~Dern, Der Hermitesche Maa\ss -Raum zum Zahlk\"{o}rper $\mathbb{Q}(i\sqrt{3})$,
Diplomarbeit, Aachen, 1996.

\bibitem{D-K} T.~Dern, A.~Krieg, Graded rings of Hermitian modular forms of degree 2, manuscripta math. \textbf{110} (2003), 251-272. 

\bibitem{Ei-Za} M.~Eichler, D.~Zagier, The theory of Jacobi forms. Progress in Mathematics, 
\textbf{55}. Birkh\"{a}user, Boston (1985), v+148 pp.

\bibitem {Ha-Pa} F.~Hao, C.~Parry, The Fermat equation over quadratic fields, J. Number Theory 
\textbf{19}(1) (1984), 115-130. 

\bibitem {Ha-Pa2} F.~Hao, C.~Parry, Generalized Bernoulli numbers and $w$-regular primes, Math.
of Computation {\bf 43}(167) (1984), 273-288.

\bibitem{Kat1} H.~Katsurada, Congruence between Duke-Imamoglu-Ikeda lifts and non-Duke-Imamogle-Ikeda lifts. preprint 2011. arXiv:1101.3377v1 [math.NT]

\bibitem{Kat2} H.~Katsurada, Congruence of Siegel modular forms and special values of their standard zeta functions. Math. Z. \textbf{259} (2008), 97-111.

\bibitem{Kat-Miz} H.~Katsurada, S.~Mizumoto, Congruences for Hecke eigenvalues of Siegel modular Forms. preprint 2012.

\bibitem{Kiku-Nag} T.~Kikuta, S.~Nagaoka,
On Hermitian modular forms mod $p$.
J. Math. Soc. Japan \textbf{63} (2011), 211-238.

\bibitem{Kri} A.~Krieg, The Maass spaces on the Hermitian half-space of degree 2, Math. Ann. 
\textbf{289} (1991), 663-681.

\bibitem{Kuro1} N.~Kurokawa, Congruences between Siegel modular forms of degree two. 
Proc. Japan Acad., Ser. A, \textbf{55} (1979), 417-422.

\bibitem{Kuro2}N.~Kurokawa, Congruences between Siegel modular forms of degree two, II. 
Proc. Japan Acad., Ser. A, \textbf{57} (1981), 140-145.

\bibitem{Kuro3} N.~Kurokawa, On Siegel eigenforms, Proc. Japan Acad. Ser. A, 
\textbf{57} (1981), 47-50. 


\bibitem {Ma} H.~Maass, Die Fourierkoeffizienten der Eisensteinreihen zweiten grades, Mat. Fys. Medd. Dan. Vid. Selsk. {\bf 34}(7) (1964), 1-25.

\bibitem{Miz} S.~Mizumoto, Congruences for eigenvalues of Hecke operators on Siegel modular forms of degree two.
Math. Ann. \textbf{275} (1986), 149-161.

\bibitem {Swi} H.P.F.~Swinnerton-Dyer, On $l$-adic representations and congruences for coefficients of modular forms. In Modular functions of one variable, III (Proc. Internat. Summer School, Univ. Antwerp, 1972), 1-55. Lecture Notes in Math. 350. Springer Verlag, Berlin, 1973.

\end{thebibliography}
\end{document}